\documentclass[11pt]{amsart}
\usepackage{NahmYmain}

\begin{document}
\title{Nahm's conjecture and Y-systems}
\author{Chul-hee Lee}
\address{Max-Planck-Institut f\"ur Mathematik, Vivatsgasse 7, 53111 Bonn, Germany}
\email{chlee@mpim-bonn.mpg.de}
\date{Sep 2011, last modified on \today}
\maketitle

\begin{abstract}
\noindent
Nahm's conjecture relates $q$-hypergeometric modular functions to torsion elements in the Bloch group. An interesting class of such functions can be (conjecturally) obtained from a pair $(X,X')$ of diagrams, each of which is either a Dynkin diagram of type $ADE$ or a diagram of type $T$. Using properties of Y-systems, we prove that for a matrix of the form $A=\mathcal{C}(X)\otimes \mathcal{C}(X')^{-1}$ where $\mathcal{C}(X)$ and $\mathcal{C}(X')$ are the corresponding Cartan matrices, every solution of the equation $\mathbf{x}=(1-\mathbf{x})^A$ gives rise to a torsion element of the Bloch group.
\end{abstract}
\maketitle
\section{Introduction}
In \cite{Nahm}, Nahm considered a question of when an $r$-fold $q$-hypergeometric series is modular and made a conjecture relating this question to algebraic K-theory, motivated by integrable perturbations of rational conformal field theories. 

\begin{definition}
Let $A$ be a positive definite symmetric $r\times r$ matrix, $B$ be a vector of length $r$, and $C$ a scalar, all three with rational entries. We consider an $r$-fold $q$-hypergeometric series
\begin{equation*}
f_{A,B,C}(z)=\sum_{n=(n_1,...,n_r)\in (\mathbb{Z}_{\geq 0})^r}\frac{q^{\frac{1}{2}n^tAn+B^tn+C}}{(q)_{n_1}\cdots(q)_{n_r}}
\end{equation*}
where $q=e^{2\pi i z}$ and $(q)_n=(1-q)(1-q^2)...(1-q^n)$. If $f_{A,B,C}$ is a modular function, then we call $(A,B,C)$ a modular triple and $A$ the matrix part of it.
\end{definition}

The most famous examples are the Rogers-Ramanujan identities:
$$\sum_{n=0}^\infty \frac {q^{n^2-1/60}} {(q)_n} = \frac {q^{-1/60}}{(q;q^5)_\infty (q^4; q^5)_\infty},$$
$$\sum_{n=0}^\infty \frac {q^{n^2+n+11/60}} {(q)_n} =\frac { q^{11/60}}{(q^2;q^5)_\infty 
(q^3; q^5)_\infty}$$
where $(a;x)_{\infty} = \prod_{k=0}^{\infty} (1-ax^k).$ In this case, we have modular triples $((2), (0),-1/60)$ and $((2), (1),11/60)$. 

As an attempt to characterize matrix parts of modular triples, one considers the asymptotic behavior of $f_{A,B,C}$ when $z$ approaches 0 and is led to a system of equations associated to the matrix $A=(a_{ij})$ given by
\begin{equation} \label{TBAeqn}
x_i=\prod_{j=1}^{r}(1-x_j)^{a_{ij}}, \quad (i=1,...,r).
\end{equation}
This is a system of $r$ equations of $r$ variables $x_1,\cdots, x_r$. When there is no confusion, we will denote this system of equations by $\mathbf{x}=(1-\mathbf{x})^A$.

For a solution $\mathbf{x}=(x_1,...,x_r)\in \mathbb{\overline{Q}}^r$ of $\mathbf{x}=(1-\mathbf{x})^A$, we consider a formal sum $\xi_{\mathbf{x}}=[x_1]+...+[x_r]$ in the group ring $\mathbb{Z}[{F}]$ of the number field $F=\mathbb{Q}(x_1,\cdots,x_r)$. It is an element of the Bloch group $\mathcal{B}(F)$ whose definition is given in Section \ref{BG}.

Nahm's conjecture is as follows : 
\begin{conjecture}
Let $A$ be a positive definite symmetric $r\times r$ matrix with rational entries. The following are equivalent: 
\begin{enumerate}
\item[(i)] For any solution $\mathbf{x}=(x_1,...,x_r)\in \mathbb{\overline{Q}}^r$ of (\ref{TBAeqn}), the element $\xi_{\mathbf{x}}$ is a torsion element of $\mathcal{B}({F})$.
\item[(ii)] There exists a modular triple $(A,B,C)$.
\end{enumerate}
\end{conjecture}

In addition to that, we expect that for such matrices $A$, there exist modular triples $(A,B_s,C_s), s\in S$ indexed by a finite set $S$ and $(f_{A,B_s,C_s})_{s\in S}$ spans a vector space which is invariant under the whole group $\rm{SL}(2,\mathbb{Z})$ of the modular transformations.

This conjecture is discussed in detail in Nahm's paper \cite{Nahm} from the viewpoint of conformal field theory and in Zagier's \cite{Zagier} from that of number theory. Zwegers and Vlasenko found counterexamples to the above conjecture in \cite{ZV} : there exists a modular triple $(A,B,C)$ for which not all solutions of (\ref{TBAeqn}) give torsion elements. Thus the conjecture must be modified and reformulated. We still expect that for a given modular triple, there exists a solution giving torsion elements in $\mathcal{B}({F})$ and (i)$\Rightarrow$(ii) is still conjectural.

It has been long known (and conjectured) that there is an interesting class of modular triples whose matrix part is given by the Kronecker product of a pair of certain matrices from Lie theory. To be precise, let us consider a matrix of the form $A=\mathcal{C}(X)\otimes \mathcal{C}(X')^{-1}$ where each of $X$ and $X'$ is either a Dynkin diagram of type $ADE$ or a diagram of type $T$ ($ADET$ diagram\footnote{Although the diagrams of type $T$ are not considered in standard Lie theory, they have been commonly used in literature related to our topic under discussion. See \cite{MR1216231} and \cite[Section 4]{Nahm} for example.} for short) and $\mathcal{C}(X)$ and $\mathcal{C}(X')$ their Cartan matrices. We denote their index sets by $I$ and $I'$ and let $\mathbf{I}=I\times I'$. The diagram of type $T_n$ can be obtained by folding the diagram of type $A_{2n}$ in the middle and $\mathcal{C}(T_n)$ is the inverse of the integral matrix $\left(\min(i,j)\right)_{1\leq i,j \leq n}$ up to a permutation of the index set $\{1,\cdots, n\}$. See Section \ref{YS2} for more about foldings of simply-laced Dynkin diagrams.

For such matrices $A$, many modular triples have been found. In the case of the Rogers-Ramanujan identities, the matrix part of the modular triples is $(2)=\mathcal{C}(A_1)\otimes \mathcal{C}(T_1)^{-1}$. The Andrews-Gordon identities are well-known generalizations of the Rogers-Ramanujan identities and they give examples of modular triples with matrix parts of the form $A=\mathcal{C}(A_1)\otimes \mathcal{C}(T_n)^{-1}$. See \cite{MR2869652} for a detailed discussion of more examples.

In this paper, we give a proof of the following theorem.
\begin{theorem}  \label{thmmain}
Let $A=\mathcal{C}(X)\otimes \mathcal{C}(X')^{-1}$ where each of $X$ and $X'$ is either a Dynkin diagram of type $ADE$ or a diagram of type $T$. For every solution $\mathbf{x}=(x_{\mathbf{i}})_{\mathbf{i}\in \mathbf{I}}$ of the equation $\mathbf{x}=(1-\mathbf{x})^A$, $\xi_{\mathbf{x}}=\sum_{\mathbf{i}\in \mathbf{I}} [x_{\mathbf{i}}]$ is a torsion element of the Bloch group $\mathcal{B}({F})$ where $F$ is the number field generated by $\mathbf{x}$. 
\end{theorem}

The proof is obtained using properties of $Y$-systems whose definition is given in Section \ref{YS}. Frenkel and Szenes studied dilogarithm identities and their relation to torsion elements in algebraic K-theory \cite{FS1} and $Y$-systems \cite{FS2}. In \cite[Section 4]{Nahm}, Nahm briefly explains how one can obtain a proof of the above statement assuming the periodicity of $Y$-systems. It seems that, however, more structural properties of $Y$-systems need to be used to complete the proof.  Nakanishi's paper \cite{Nakanishi} contains most of results used here except relating results to the Bloch group.

The $Y$-system, which can be defined for a pair of Dynkin diagrams, turns out to be very useful to study (\ref{TBAeqn}) as we will see in Proposition \ref{PROPY}.  We can relate the equation to the $Y$-system and then using properties of $Y$-systems, we can show that all solutions give torsion elements of the Bloch group. 

Many conjectured properties of $Y$-systems such as periodicities and functional dilogarithm identities, originated from thermodynamic Bethe ansatz approach of conformal field theory \cite{Zam,MR1216231, MR1325407}, had remained open for years but now have been proved rigorously due to recent development of the theory of cluster algebras. See \cite{FZ} and \cite{Keller}. One may hope that a correct reformulation of Nahm's conjecture incorporates this and it would help us to find new directions toward understanding modular triples and modular $q$-hypergeometric series.

In Sections \ref{BG} and \ref{YS}, we give necessary definitions and properties of the Bloch group and $Y$-systems. We prove Theorem \ref{thmmain} in Section \ref{PF}. 

\section{the Bloch group} \label{BG}
In this section, we give the definition of the Bloch group for a field and explain the role of the Bloch-Wigner dilogarithm function in its study. See \cite{Zagier} and \cite{Zagier2} for a more thorough discussion of the topic.
\begin{definition}
Let $F$ be a field and $\Lambda^2 F^{\times}$ be the abelian group of formal sums of $x\wedge y, x,y\in F^{\times}$
modulo the relations $x\wedge x=0, (x_1x_2)\wedge y=x_1\wedge y+x_2\wedge y$ and $x\wedge (y_1y_2)=x\wedge y_1+x\wedge y_2$. 

Let $\partial : \mathbb{Z}[{F^{\times}\backslash \{1\}}]\to \Lambda^2({{F}^{\times}})$ be a $\mathbb{Z}$-linear map defined by $\partial([x])=x\wedge (1-x)$.
Let $A({F})=\operatorname{ker}\partial$ and $C({F})$ the subgroup of $A({F})$ generated by the elements
\begin{equation} \label{fiveterm}
[x]+[1-xy]+[y]+[\frac{1-y}{1-xy}]+[\frac{1-x}{1-xy}],
\end{equation}
\begin{equation*}
[x]+[1-x]\text{ and }[x]+[\frac{1}{x}].
\end{equation*}

It is convenient to set $[0]=[1]=[\infty]=0$ in $A(F)$. We call (\ref{fiveterm}) the five-term relation. The Bloch group  $\mathcal{B}({F})$ of ${F}$ is defined to be $A({F})/C({F})$. 
\end{definition}

The dilogarithm function is defined by 
$$\operatorname{Li}_2(z) = -\int_0^z{{\log (1-t)}\over t} dt$$
for $z\in \mathbb C-[1,\infty)$. For $|z|<1$, we have the following power series expansion:
$$\operatorname{Li}_2(z)= \sum_{n=1}^\infty {z^n \over n^2}.$$

Let us define a variant of the dilogarithm function : the Bloch-Wigner dilogarithm function. It is given by
$$D(z)=\text{Im}(\operatorname{Li}_2(z))+\log|z|\arg(1-z).$$
It is a real analytic function on $\mathbb{C}$ except at 0 and 1, where it is continuous but not differentiable. Since $D(\bar{z})=-D(z)$, it vanishes on $\mathbb{R}$. 
It satisfies the following functional equations :
\begin{equation}\label{functid1}
D(x)+D(1-xy)+D(y)+D(\frac{1-y}{1-xy})+D(\frac{1-x}{1-xy})=0,
\end{equation}
\begin{equation}\label{functid2}
D(x)+D(1-x) =D(x)+D(\frac{1}{x})=0.
\end{equation}

The Bloch-Wigner dilogarithm $D(z)$ can be used to define a map from $\mathcal{B}(\mathbb{C})$ to $\mathbb{R}$. For $\xi=\sum_{i} n_i[x_i] \in \mathcal{B}(\mathbb{C})$, let $D(\xi)=\sum_{i} n_i D(x_i)$. By (\ref{functid1}) and (\ref{functid2}), it is well-defined. Let $F$ be a number field of degree $r_1+2r_2$ over $\mathbb{Q}$ where $r_1$ denotes the number of real embeddings and $r_2$ the number of complex non-real embeddings up to conjugation. For an embedding $\sigma : F\hookrightarrow \mathbb{C}$ and $\xi \in \mathcal{B}(F)$, we may consider  $D\left(\sigma(\xi)\right)$. If $D\left(\sigma(\xi)\right)=0$ for all such embeddings $\sigma$, then $\xi \in \mathcal{B}(F)$ is a torsion element  in $\mathcal{B}(F)$. This is a consequence of the known isomorphism between $K_3(F)\otimes \mathbb{Q}$ and $\mathcal{B}(F)\otimes \mathbb{Q}$ and of Borel's description of $K_3(F)$ modulo torsion; for more details and references, see \cite[Section 2]{Zagier2}.

\begin{proposition}\cite{FS2, MR1399471} \label{thmgon}
Let $F=\mathbb{C}(y_1,\cdots,y_n)$ be a field of rational functions. Given an $\sum_{i}n_i [f_i]\in \mathbb{Z}[F]$ such that $\sum_{i}n_i\left(f\wedge (1-f)\right)=0$ in $\Lambda^2 F^{^\times}$, the function
$z\mapsto \sum_{i} n_iD\left(f_i(z)\right)$ from $\mathbb{C}^n$ to $\mathbb{R}$ is constant.
\end{proposition}

See \cite[Chapter II. Section 2.A]{Zagier} for a short proof and references. In order to obtain such a set of rational functions satisfying the condition of the above statement, we now turn our attention to $Y$-systems. They are good suppliers for such rational functions as we can see in Proposition \ref{torsion}.

\section{$Y$-systems} \label{YS} 
\subsection{The $Y$-system for a pair of $ADE$ Dynkin diagrams}
In this section, we closely follow the notations of \cite{Nakanishi}. Let $X$ be a Dynkin diagram of type $ADE$ with the index set $I$.  The rank and the Coxeter number of $X$ will be denoted by $r$ and $h$. We denote the Cartan matrix of $X$ by $\mathcal{C}(X)$ and the adjacency matrix by $\mathcal{I}(X)=2I_r-\mathcal{C}(X)$ where $I_r$ is the identity matrix of size $r$. We call a decomposition $I=I_{+}\cup I_{-}$ bipartite if $\mathcal{I}(X)_{ij}=1$ implies $(i,j)\in I_{+}\times I_{-}$ or $(i,j)\in I_{-}\times I_{+}$.
Now consider an ordered pair of Dynkin diagrams $(X,X')$. For another Dynkin diagram $X'$, $I'=I'_{+}\cup I'_{-}$, $r'$, $h'$, $\mathcal{C}(X')$, and $\mathcal{I}(X')$ will be defined analogously.

We give an alternate bicoloring on the pair of Dynkin diagrams. Let us fix bipartite decompositions of $I$ and $I'$. Let $ \mathbf{I}= I\times I'$ and $\mathbf{I}=\mathbf{I}_{+}\sqcup \mathbf{I}_{-}$ where $\mathbf{I}_{+}=(I_{+}\times I'_{+}) \sqcup (I_{-}\times I'_{-})$ and $\mathbf{I}_{-}=(I_{+}\times I'_{-}) \sqcup (I_{-}\times I'_{+})$. Let $\epsilon : \mathbf{I}\to \{1,-1\}$ be the function defined by $\epsilon(\mathbf{i})=\pm 1$ for $\mathbf{i}\in \mathbf{I}_{\pm}$ and $P_{\pm} =\{(\mathbf{i},u)\in \mathbf{I}\times\mathbb{Z}| \epsilon(\mathbf{i})(-1)^u=\pm 1\}$. Roughly speaking, we want our alternate bicoloring interchanges their colors as $u\in \mathbb{Z}$ changes by 1. 

\begin{definition}
For a family of variables, $\{ Y_{ii'}(u)|i\in I, i'\in I', u\in \mathbb{Z}\}$, the $Y$-system $\mathbb{Y}(X,X')$ associated with a pair $(X,X')$ of Dynkin diagrams of type $ADE$ is defined as a system of recurrence relations as follows : 
$$
Y_{ii'}(u-1)Y_{ii'}(u+1)=\frac{\prod_{j:j\sim i} (1+Y_{ji'}(u))}{\prod_{j':j'\sim i'} (1+Y_{ij'}(u)^{-1})}
$$
where $a \sim b$ means $a$ is adjacent to $b$.
\end{definition}
Note that $\mathbb{Y}(X,X')$ consists of two decoupled copies, $\{Y_{\mathbf{i}}(u) | (\mathbf{i},u)\in P_{+}\}$ and $\{Y_{\mathbf{i}}(u) | (\mathbf{i},u)\in P_{-}\}$. 
If $(\mathbf{i},u)\in P_{+}$, $Y_{\mathbf{i}}(u)$ can be written as a rational function of variables $\{Y_{\mathbf{i}}(0) |\mathbf{i}\in \mathbf{I}_{+}\}$ and $\{Y_{\mathbf{i}}(-1)|\mathbf{i}\in \mathbf{I}_{-}\}$ 
whereas if $(\mathbf{i},u)\in P_{-}$, $Y_{\mathbf{i}}(u)$ only depends on $\{Y_{\mathbf{i}}(0) |\mathbf{i}\in \mathbf{I}_{-}\}$ and $\{Y_{\mathbf{i}}(-1)|\mathbf{i}\in \mathbf{I}_{+}\}$.

\begin{example} \label{4:Y21example}
Let us consider the example of $\mathbb{Y}(A_2,A_1)$. Since the index set $I'$ of the $A_1$ Dynkin diagram consists of the single element 1, we just set $Y_{i,1}=Y_i$ for $i=1,2$. The recurrence relation of the $Y$-system is
$$Y_{i}(u-1)Y_{i}(u+1)=\prod_{j:j\sim i} (1+Y_{j}(u)).$$ If we write the sequence explicitly, we get the following : 
{\small
$$
\begin{array}{c|c|c|c|c|c|c|c|c|c|c|c|c|c}
u & -1 & 0 & 1 & 2 & 3 & 4 & 5 & \cdots & 10 & 11 & \cdots \\
\hline
Y_1(u) & \frac{1}{\alpha } & y & \alpha  (\beta +1) & \frac{x+y+1}{x y} &
   \frac{\alpha +1}{\alpha  \beta } & x & \beta  & \cdots & \frac{1}{\alpha
   } & y & \cdots \\
\hline
Y_2(u) & x & \beta  & \frac{y+1}{x} & \frac{\beta  \alpha +\alpha +1}{\beta } &
   \frac{x+1}{y} & \frac{1}{\alpha } & y & \cdots & x & \beta  &
   \cdots \\
\end{array}.
$$
}

We can clearly observe the decoupling of the $Y$-system and that it is a periodic sequence of period 10. Another important thing to note is that all terms are Laurent polynomials of the initial conditions. 
\end{example}

\begin{definition}
If a solution $\{Y_{\mathbf{i}}(u)|\mathbf{i}\in \mathbf{I}, u\in \mathbb{Z}\}$ of $\mathbb{Y}(X,X')$ does not have any dependence on $u$ so that $Y_{\mathbf{i}}(u)=y_{\mathbf{i}}$ for each $\mathbf{i}$ in a field, it must satisfy the following system of $rr'$ equations of $rr'$ variables :
$$
y_{ii'}^2=\frac{\prod_{j:j\sim i} (1+y_{ji'})}{\prod_{j':j'\sim i'} (1+y_{ij'}^{-1})}. 
$$
We call it the constant $Y$-system and denote it by $\mathbb{Y}_{c}(X,X')$.
\end{definition}
One can see the importance of the constant $Y$-system in Proposition \ref{PROPY} in our study.

\subsection{Properties of $Y$-systems}
Now we state several important results about $Y$-systems. They will be used in our proof of Theorem \ref{thmmain}. For all theorems below, we assume that $\{Y_{\mathbf{i}}(u)|\mathbf{i}\in \mathbf{I}, u\in \mathbb{Z}\}$ satisfies the $Y$-system $\mathbb{Y}(X,X')$ associated to a pair of $ADE$ Dynkin diagrams.

\begin{theorem} \cite{Keller}\label{thmperiod} For any $\mathbf{i}\in \mathbf{I}$, 
$$
Y_{\mathbf{i}}(u+2(h+h'))=Y_{\mathbf{i}}(u).
$$
\end{theorem}

Let $(y_{\mathbf{i}})_{\mathbf{i}\in \mathbf{I}}$ be indeterminates and set
\begin{alignat*}{2}
Y_{\mathbf{i}}(0) = y_{\mathbf{i}}, \mathbf{i} \in \mathbf{I}_{+}, \\
Y_{\mathbf{i}}(-1) = y_{\mathbf{i}}^{-1},\mathbf{i}\in \mathbf{I}_{-}.
\end{alignat*}

Then each $Y_{\mathbf{i}}(u)$ with $(\mathbf{i},u)\in P_{+}$ can be regarded as a rational function in $y_{\mathbf{i}}$'s. Let $\mathbb{Q}(y)$ be the field of rational functions in $y_{\mathbf{i}}$'s. For $f\in \mathbb{Q}(y)$, $f|_{\mathbf{a}}$ denotes the evaluation of $f$ at $(y_{\mathbf{i}})=\mathbf{a}=(a_{\mathbf{i}})\in \mathbb{C}^{n}$ where $n=rr'$. 

\begin{theorem}\cite{Nakanishi}\label{thmsgn}
For $(\mathbf{i},u)\in P_{+}$, 
$Y_{\mathbf{i}}(u)=G_{\mathbf{i}}(u)T_{\mathbf{i}}(u)\in \mathbb{Q}(y)$
where $G_{\mathbf{i}}(u)\in \mathbb{Q}(y)$ satisfies $G_\mathbf{i}(u)|_{(0,\cdots,0)}=1$ and $T_{\mathbf{i}}(u)\neq 1$ is a positive or negative monomial in $y_{\mathbf{i}}$'s, i.e. $T_{\mathbf{i}}(u)$ can be written as a product of $y_{\mathbf{i}}$'s or as a product of $y_{\mathbf{i}}^{-1}$'s.
\end{theorem}

We state a property of the $Y$-system, which Nakanishi called the constancy condition. 
\begin{theorem}\cite[Proposition 3.2 (i)]{Nakanishi} \label{thmconst}
The following property holds :
$$
\sum_{(\mathbf{i},u)\in S_{+}} Y_{\mathbf{i}}(u)\wedge (1+Y_{\mathbf{i}}(u))=0\in \Lambda^2 \mathbb{Q}(y)^{\times}
$$ where $S_{+}=\{(\mathbf{i},u) |0\leq u \leq 2(h+h')-1,(\mathbf{i},u)\in P_{+}\}$.
\end{theorem}

\subsection{The $Y$-system for a pair of foldings of $ADE$ Dynkin diagrams} \label{YS2} 
The results in the previous section can be extended to include all foldings of $ADE$ diagrams almost trivially.

First note that for a pair  $(X,X')$ of directed graphs with the index sets $I$ and $I'$, we can redefine the $Y$-system in terms of their adjacency matrices of graphs as follows : 

\begin{equation}
Y_{ii'}(u-1)Y_{ii'}(u+1)=\frac{\prod _{j\in I} (1+Y_{ji'}(u))^{\mathcal{I}(X)_{ij}}}{\prod _{j'\in I'} \left(1+Y_{ij'}(u)^{-1}\right)^{\mathcal{I}(X')_{i'j'}}}.
\end{equation}

Let $X$ be a Dynkin diagram of type $ADE$. Let us pretend to think that there are directed edges $(i,j)$ and $(j,i)$ for the edge connecting $i$ and $j$ in $X$ and then we can regard it as a directed graph with the adjacency matrix $\mathcal{I}(X)$.
For a group $G$ of diagram automorphisms of $X$, we can define a quotient diagram $\bar{X}=X/G$ as follows : $\bar{X}$ has the vertex set $\bar{I}$, the orbit of $I$ under $G$ and the ordered pair $(\bar{i},\bar{j})$ is an edge of $\bar{X}$ if $(i,j)$ is an edge of $G$ and its multiplicity $\mathcal{I}(\bar{X})_{\bar{i}\bar{j}}$ is defined as the number of preimages of $\bar{j}$  for a fixed representative $i$ of $\bar{i}$. Let us call $\bar{X}$ the folding of $X$ by $G$. 

Note that $\bar{X}$ is generally a directed graph and its adjacency matrix $\mathcal{I}(\bar{X})$ may not be symmetric. Let us call the matrix $\mathcal{C}(\bar{X})=2I_r-\mathcal{I}(\bar{X})$ the Cartan matrix of $\bar{X}$. If $G$ is a trivial group, we just get $\bar{X}=X$. The Coxeter number of $\bar{X}$ is the same as the Coxeter number of $X$. 

The tadpole graph $T_{r}$ is obtained as the folding of $X=A_{2r}$ by the diagram automorphism group of order 2. The Cartan matrix $\mathcal{C}(T_r)$ is the same as $\mathcal{C}(A_r)$ except that the diagonal entry corresponding to the vertex with a loop is 1 instead of 2 because $T_r$ diagram has a loop.

The folding $\bar{X}$ inherits the bipartite decomposition $\bar{I}=\bar{I}_{+}\cup \bar{I}_{-}$ of $X$ except when $\bar{X}=T_n$ because $T_n$ diagram has a loop and cannot be bipartite, in which case, we just set $\bar{I}={\bar{I}}_{+} ={\bar{I}}_{-}$.

Let $(\bar{X},\bar{X'})$ be a pair of foldings of $ADE$ Dynkin diagrams. For $ \mathbf{\bar{I}}= \bar{I}\times \bar{I}'$ we define a decomposition $\mathbf{\bar{I}}=\mathbf{\bar{I}}_{+}\cup \mathbf{\bar{I}}_{-}$ where $\mathbf{\bar{I}}_{+}=(\bar{I}_{+}\times \bar{I}'_{+}) \cup (\bar{I}_{-}\times \bar{I}'_{-})$ and $\mathbf{\bar{I}}_{-}=(\bar{I}_{+}\times \bar{I}'_{-}) \cup (\bar{I}_{-}\times \bar{I}'_{+})$. When $\bar{X}=T_{r}$ or  $\bar{X'}=T_{r'}$, we just get $\bar{\mathbf{I}}=\bar{\mathbf{I}}_{+}= \bar{\mathbf{I}}_{-}$. When $\bar{X}=T_{r}$ or  $\bar{X'}=T_{r}$, we also set ${\bar{P}}_{+} =\bar{P}_{-} =  \mathbf{\bar{I}}\times\mathbb{Z}$. Otherwise, we can define $\bar{P}_{\pm}$ similarly as in the previous section. 

Now we extend the theorems in the previous subsection to the Y-system associated to a pair $(\bar{X},\bar{X'})$ of foldings of $ADE$ Dynkin diagrams. For the rest of this subsection, let $\left(Y_{\bar{\mathbf{i}}}(u)\right)_{(\bar{\mathbf{i}},u)\in \mathbf{\bar{I}}\times \mathbb{Z}}$ be a solution of $\mathbb{Y}(\bar{X},\bar{X'})$. Note that if $\left(Y_{\bar{\mathbf{i}}}(u)\right)_{(\bar{\mathbf{i}},u)\in \mathbf{\bar{I}}\times \mathbb{Z}}$ is a solution of $\mathbb{Y}(\bar{X},\bar{X'})$, then we can obtain a solution $\left(Y_{\mathbf{i}}(u)\right)_{(\mathbf{i},u)\in \mathbf{I}\times \mathbb{Z}}$ of $\mathbb{Y}(X,X')$ by simply setting $Y_{\mathbf{i}}(u):=Y_{\bar{\mathbf{i}}}(u)$ for each $(\mathbf{i},u)\in \mathbf{I}\times \mathbb{Z}$. Thus, a solution of $\mathbb{Y}(\bar{X},\bar{X'})$ is nothing but a solution of $\mathbb{Y}(X,X')$ with symmetries given by the group $G$.

As before, let $(y_{\bar{\mathbf{i}}})_{\bar{\mathbf{i}}\in \bar{\mathbf{I}}}$ be indeterminates and set
\begin{alignat*}{2}
Y_{\bar{\mathbf{i}}}(0) = y_{\bar{\mathbf{i}}}, \bar{\mathbf{i}} \in \bar{\mathbf{I}}_{+}, \\
Y_{\bar{\mathbf{i}}}(-1) = y_{\bar{\mathbf{i}}}^{-1}, \bar{\mathbf{i}} \in \bar{\mathbf{I}}_{-}.
\end{alignat*}
Then again each $Y_{\bar{\mathbf{i}}}(u)$ with $(\bar{\mathbf{i}},u)\in {\bar{P}}_{+}$ can be regarded as an element of $\mathbb{Q}(y)$, the field of rational functions in $y_{\bar{\mathbf{i}}}$'s. 

\begin{theorem} \label{thmsgn2}
Theorem \ref{thmperiod} and \ref{thmsgn} hold true for a solution of $\mathbb{Y}(\bar{X},\bar{X'})$. 
\end{theorem}

Let $\bar{S}_{+}=\{(\bar{\mathbf{i}},u)\in \bar{P}_{+}|0\leq u \leq 2(h+h')-1\}$.
For each $(\bar{\mathbf{i}},u)\in \bar{S}_{+}$, let  $d_{\bar{\mathbf{i}}}(u)$ be the number of preimages of $(\bar{\mathbf{i}},u)$ under the quotient map $S_{+} \to \bar{S}_{+}$ given by $(\mathbf{i},u) \mapsto (\bar{\mathbf{i}},u)$. Then Theorem \ref{thmconst} can be restated as follows :

\begin{theorem} \label{thmconst2}
If $\left(Y_{\bar{\mathbf{i}}}(u)\right)_{(\bar{\mathbf{i}},u)\in \mathbf{\bar{I}}\times \mathbb{Z}}$ is a solution of $\mathbb{Y}(\bar{X},\bar{X'})$, then
$$
\sum_{(\bar{\mathbf{i}},u)\in \bar{S}_{+}} d_{\bar{\mathbf{i}}}(u)(Y_{\bar{\mathbf{i}}}(u)\wedge (1+Y_{\bar{\mathbf{i}}}(u)))=0\in \Lambda^2 \mathbb{Q}(y)^{\times}.
$$
In other words, the element
$$
\sum_{(\bar{\mathbf{i}},u)\in \bar{S}_{+}}  d_{\bar{\mathbf{i}}}(u) [\frac{Y_{\bar{\mathbf{i}}}(u)}{1+Y_{\bar{\mathbf{i}}}(u)}]
$$
of the group ring of  $\mathbb{Q}(y)$ is an element of the Bloch group $\mathcal{B}(\mathbb{Q}(y))$.
\end{theorem}

Now we prove that the $Y$-system produces a torsion element of the Bloch group.
\begin{proposition} \label{torsion}
Let $f_{\bar{\mathbf{i}}}(u)=\frac{Y_{\bar{\mathbf{i}}}(u)}{1+Y_{\bar{\mathbf{i}}}(u)}\in \mathbb{Q}(y)$. Then
$$
\sum_{(\bar{\mathbf{i}},u)\in \bar{S}_{+}} d_{\bar{\mathbf{i}}}(u) D(f_{\bar{\mathbf{i}}}(u)|_{\mathbf{x}})=0 
$$
for any $\mathbf{x}=(x_{\bar{\mathbf{i}}})\in \mathbb{C}^{n}$ where $n=rr'$. 
\end{proposition}

\begin{proof}
To employ Proposition \ref{thmgon}, we check the following condition
\begin{equation} \label{cond1}
\sum_{(\bar{\mathbf{i}},u)\in \bar{S}_{+}} d_{\bar{\mathbf{i}}}(u) \left(f_{\bar{\mathbf{i}}}(u)\wedge (1-f_{\bar{\mathbf{i}}}(u))\right)=0.
\end{equation}
This is equivalent to 
$$\sum_{(\bar{\mathbf{i}},u)\in \bar{S}_{+}} d_{\bar{\mathbf{i}}}(u) \left((\frac{Y_{\bar{\mathbf{i}}}(u)}{1+Y_{\bar{\mathbf{i}}}(u)})\wedge(\frac{1}{1+Y_{\bar{\mathbf{i}}}(u)})\right)=0,$$
which reduces to the constancy condition of the $Y$-system,
$$\sum_{(\bar{\mathbf{i}},u)\in \bar{S}_{+}} d_{\bar{\mathbf{i}}}(u) \left(Y_{\bar{\mathbf{i}}}(u)\wedge (1+Y_{\bar{\mathbf{i}}}(u))\right)=0.$$
Thus (\ref{cond1}) is satisfied by Theorem \ref{thmconst2}.

Now all we have to check is that there is a point $
\mathbf{a}\in \mathbb{C}^n$  such that $\sum_{(\bar{\mathbf{i}},u)\in \bar{S}_{+}}D(f_{\bar{\mathbf{i}}}(u)|_{\mathbf{a}})=0$.
By Theorem \ref{thmsgn2}, $f_{\bar{\mathbf{i}}}(u)=\frac{G_{\bar{\mathbf{i}}}(u)T_{\bar{\mathbf{i}}}(u)}{1+G_{\bar{\mathbf{i}}}(u)T_{\bar{\mathbf{i}}}(u)}$. 
Since $G_{\bar{\mathbf{i}}}(u)|_{(0,\cdots,0)}=1$ and $T_{\bar{\mathbf{i}}}(u)\neq 1$ is a positive or negative monomial in $y_{\bar{\mathbf{i}}}$'s, $f_{\bar{\mathbf{i}}}(u)|_{(0,\cdots,0)}$ is always 0 or 1 depending on whether $T_{\bar{\mathbf{i}}}(u)$ is positive or negative. We simply choose $\mathbf{a}=(0,\cdots,0)$ and get $\sum_{(\bar{\mathbf{i}},u)\in \bar{S}_{+}}D(f_{\bar{\mathbf{i}}}(u)|_{\mathbf{a}})=0$. Therefore $\sum_{(\bar{\mathbf{i}},u)\in \bar{S}_{+}}D(f_{\bar{\mathbf{i}}}(u)|_{\mathbf{x}})=0$ for any  $\mathbf{x}=(x_{\bar{\mathbf{i}}})\in \mathbb{C}^n$ by Proposition \ref{thmgon}. 
\end{proof}

\begin{remark}
This proposition generalizes \cite[Theorem 2]{FS2} and \cite[Corollary 6.14]{FG}. If we count how many of $f_{\mathbf{i}}(u)|_{(0,\cdots,0)}$ becomes 1 in the above argument, we can obtain a proof of the dilogarithm identities for central charges of certain conformal field theories \cite{Nakanishi}. 
\end{remark}

\begin{corollary} \label{torsion2}
Let $(\bar{X},\bar{X'})$ be a pair of $ADET$ diagrams. If $\mathbf{y}=(y_{\bar{\mathbf{i}}})$ is a solution of the constant $Y$-system $\mathbb{Y}_{c}(\bar{X},\bar{X'})$, then
$$
\sum_{\bar{\mathbf{i}}\in \bar{\mathbf{I}}} [\frac{y_{\bar{\mathbf{i}}}}{1+y_{\bar{\mathbf{i}}}}] \in \mathcal{B}(F)
$$
is a torsion element of the Bloch group $\mathcal{B}(F)$ where $F$ is the number field generated by $\mathbf{y}$.
\end{corollary}

\begin{proof}
Let $\sigma : F\hookrightarrow \mathbb{C}$ be an embedding. 
By Proposition \ref{torsion}, we know
$$
\sum_{(\mathbf{i},u)\in \bar{S}_{+}}d_{\bar{\mathbf{i}}}(u) D\left(\sigma(\frac{y_{\bar{\mathbf{i}}}}{1+y_{\bar{\mathbf{i}}}})\right)=0.
$$

Note that when $(\bar{X},\bar{X'})$ is given by a pair of $ADET$ diagrams, $d_{\bar{\mathbf{i}}}(u)$ is the same for all $(\bar{i},u)$.  For $(\bar{X},\bar{X'})=(T_r,T_{r'})$, we have $d_{\bar{\mathbf{i}}}(u)=2$ and $d_{\bar{\mathbf{i}}}(u)=1$ in other cases. 

Thus we get
$$
\sum_{\bar{\mathbf{i}}\in \bar{\mathbf{I}}} D\left(\sigma(\frac{y_{\bar{\mathbf{i}}}}{1+y_{\bar{\mathbf{i}}}})\right)=0.
$$

Since this is true for any $\sigma : F\hookrightarrow \mathbb{C}$, $\sum_{\bar{\mathbf{i}}\in \bar{\mathbf{I}}} [\frac{y_{\bar{\mathbf{i}}}}{1+y_{\bar{\mathbf{i}}}}]$ is a torsion element of the Bloch group.
\end{proof}

\section{Proof of the main theorem}\label{PF}
We are now ready to prove Theorem \ref{thmmain}. Let $(X,X')$ be a pair of $ADET$ diagrams and $A=\mathcal{C}(X)\otimes \mathcal{C}(X')^{-1}$. First we relate a solution of the equation $\mathbf{x}=(1-\mathbf{x})^A$ to the constant $Y$-system $\mathbb{Y}_{c}(X,X')$. This will show that the constant $Y$-system $\mathbb{Y}_{c}(X,X')$ is just a disguised form of the equation $\mathbf{x}=(1-\mathbf{x})^A$.

\begin{proposition} \label{PROPY}
If $\mathbf{x}=(x_{\mathbf{i}})$ is a solution to $\mathbf{x}=(1-\mathbf{x})^A$ in a number field, then $\mathbf{y}=(y_{\mathbf{i}})$ where $y_{\mathbf{i}}=\frac{x_{\mathbf{i}}}{1-x_{\mathbf{i}}}$ for each $\mathbf{i}\in \mathbf{I}$ is a solution to the constant $Y$-system $\mathbb{Y}_c(X,X')$.
\end{proposition}

\begin{proof}
Let us rewrite the equation $\mathbf{x}=(1-\mathbf{x})^A$ as
\begin{equation}  \label{eq41}
x_{\mathbf{i}}=\prod _{\mathbf{j}\in\mathbf{I}} (1- x_{\mathbf{j}})^{(\mathcal{C}(X)\otimes \mathcal{C}(X')^{-1})_{\mathbf{ij}}},
\end{equation}
or,
$$
x_{ii'}=\prod _{(j,j')\in I\times I'} (1-x_{jj'})^{(\mathcal{C}(X)\otimes \mathcal{C}(X')^{-1})_{\mathbf{ij}}}.
$$
This implies
\begin{equation}\label{4:XtoXpair}
\prod _{j'\in I'} x_{ij'}^{\mathcal{C}(X')_{i'j'}}=\prod _{j\in I} (1-x_{ji'})^{\mathcal{C}(X)_{ij}}.
\end{equation}

Since $A$ is a positive definite matrix, all diagonal entries are positive. Thus from (\ref{eq41}), we can see that $x_{\mathbf{i}}$ is neither 0 nor 1. 

Now use the change of variables $y_{\mathbf{i}}=\frac{x_{\mathbf{i}}}{1-x_{\mathbf{i}}}$ or $x_{\mathbf{i}}=\frac{y_{\mathbf{i}}}{1+y_{\mathbf{i}}}=\frac{1}{1+y_{\mathbf{i}}^{-1}}$.
From (\ref{4:XtoXpair}), we obtain
$$
\prod _{j'\in I'} (\frac{1}{1+y_{ij'}^{-1}})^{\mathcal{C}(X')_{i'j'}}=\prod _{j\in I} (\frac{1}{1+y_{ji'}})^{\mathcal{C}(X)_{ij}},
$$
and thus get
$$
1= \frac{\prod _{j\in I}({1+y_{ji'}})^{-\mathcal{C}(X)_{ij}}}{\prod _{j'\in I'} ({1+y_{ij'}^{-1}})^{-\mathcal{C}(X')_{i'j'}}}.
$$
This can be written as
$$
\left(\frac{1}{1+y_{ii'}^{-1}}\right)^2(1+y_{ii'})^2=\frac{\prod _{j\in I} (1+y_{ji'})^{\mathcal{I}(X)_{ij}}}{\prod _{j'\in I'} \left(1+y_{ij'}^{-1}\right)^{\mathcal{I}(X')_{i'j'}}}.
$$
We thus have obtained the constant $Y$-system
$$
y_{ii'}^2=\frac{\prod_{j\in I} (1+y_{ji'})^{\mathcal{I}(X)_{ij}}}{\prod_{j'\in I'} \left(1+y_{ij'}^{-1}\right)^{\mathcal{I}(X')_{i'j'}}}.
$$
\end{proof}

Now we can finish the proof of Theorem \ref{thmmain}. 
\begin{proof}
Let $\mathbf{x}=(x_{\mathbf{i}})$ be a solution to $\mathbf{x}=(1-\mathbf{x})^A$. 
By Proposition \ref{PROPY}, $y_{\mathbf{i}}=\frac{x_{\mathbf{i}}}{1-x_{\mathbf{i}}}$ is a solution to the constant $Y$-system  $\mathbb{Y}_{c}(X,X')$. Then our theorem follows from Corollary \ref{torsion2}.
\end{proof}

\section*{Acknowledgements}
This is part of the author's Ph.D. thesis, written under the supervision of Richard Borcherds at the University of California, Berkeley. He wishes to thank An Huang and Richard Borcherds for helpful discussions, Nicolai Reshetikhin for a suggestion to learn cluster algebras, Edward Frenkel for an explanation of his works and Don Zagier for valuable comments on earlier drafts. This work is partially supported by Samsung Scholarship.
\bibliographystyle{amsalpha}
\bibliography{NahmYmain}
\end{document}